\documentclass[12pt,reqno]{amsart}

\usepackage[centering=true]{geometry}
\usepackage[utf8]{inputenc}
\usepackage[T1]{fontenc}
\usepackage[english]{babel}

\usepackage[colorlinks=true,linkcolor=teal,citecolor=teal]{hyperref}
\usepackage{cite}

\usepackage{mathtools}
\usepackage{amssymb}
\usepackage{amsthm}
\usepackage{nccmath}

 \usepackage{tikz}
\usepackage{graphicx}
\usepackage{xcolor}
\usepackage[makeroom]{cancel}

\theoremstyle{definition}
\newtheorem{definition}{Definition}

\newtheorem{example}{Example}

\theoremstyle{plain}
\newtheorem{theorem}{Theorem}
\newtheorem{lemma}{Lemma}
\newtheorem{corollary}{Corollary}
\newtheorem{proposition}{Proposition}

\theoremstyle{remark}
\newtheorem{remark}{Remark}

\newcommand{\norm}[1]{\left\|#1\right\|}
\newcommand{\metric}[2]{\left\langle#1,#2\right\rangle}
\newcommand{\grad}[0]{\mathrm{grad}}
\newcommand{\Hess}[0]{\mathrm{Hess}}
\newcommand{\tr}[0]{\operatorname{tr}}
\newcommand{\Ric}[0]{\operatorname{Ric}}

\makeatletter
\newcommand{\rel@kern}[1]{\kern#1\dimexpr\macc@kerna}
\newcommand{\widebar}[1]{\begingroup \def\mathaccent##1##2{\rel@kern{0.8} \overline{\rel@kern{-0.8}\macc@nucleus\rel@kern{0.2}} \rel@kern{-0.2}} \macc@depth\@ne \let\math@bgroup\@empty \let\math@egroup\macc@set@skewchar \mathsurround\z@ \frozen@everymath{\mathgroup\macc@group\relax} \macc@set@skewchar\relax \let\mathaccentV\macc@nested@a \macc@nested@a\relax111{#1} \endgroup}
\makeatother

\begin{document}

\title[Geometric Properties and Spectral Estimates on Warped Products]{Geometric Properties and Spectral Estimates on Warped Products}

\author[Mel\'endez]{Josu\'e Mel\'endez}
\address[Mel\'endez]{Departamento de Matem\'aticas, Universidad Aut\'onoma Metropolitana-Iztapalapa \\ CP 09340 \\ M\'exico City \\ M\'exico}
\email[Corresponding author]{jms@xanum.uam.mx}

\author[Rodr\'iguez-Romero]{Eduardo Rodr\'iguez-Romero}
\address[Rodr\'iguez-Romero]{Departamento de Matem\'aticas, Universidad Aut\'onoma Metropolitana-Iztapalapa \\ CP 09340 \\ M\'exico City \\ M\'exico}
\email{err29072019@gmail.com}

\author[Torres-Orozco]{Jonatán Torres Orozco}
\address[Torres-Orozco]{Departamento de Matem\'aticas, Universidad Aut\'onoma Metropolitana-Iztapalapa \\ CP 09340 \\ M\'exico City \\ M\'exico}
\email{jonatan.tto@gmail.com}



\date{\today}

\begin{abstract}
    We establish an integral inequality for the Ricci curvature of a certain class of warped products $M\times_fN$, where the equality holds if and only if it is simply a Riemannian product. We also give a sufficient condition for the intersection of a warped product $M=\mathbb{R}\times_fP$ with a totally geodesic hypersurface $N$ in an arbitrary Riemannian space to be a totally geodesic slice of $M$. In addition, we establish some spectral estimates for the Laplacian of a submanifold $N$ that intersects a warped product in the same ambient manifold.
\end{abstract}

\maketitle

\section{Introduction and main results}

One of the most useful extensions of the cartesian product of Riemannian manifolds is the notion of warped product, first defined in \cite[Section 7]{bishop-oneill}. As a generalization, the warped products have given rise to a large family of interesting and useful examples of Riemannian manifolds, including some fundamental ones in general relativity (see \cite{oneill} as a reference).
\medskip

{Furthermore, many natural examples of warped products $ {M}=I\times_f P$ arise from \emph{cohomogeneity one isometric actions}. A Lie group $G$ acts isometrically by \emph{cohomogeneity one} on a Riemannian manifold $({M},g)$ if $\dim({M}/G)=1$. 
The theory states that there are only two types of orbits: those with maximal dimension, i.e., $N-1$, called {\em principal orbits}; and those with dimension $k<N-1$, called {\em singular orbits}. If $M/G=[a, b]$, then there exists exactly two singular orbits, and only one principal orbit. Then the regular part ${M}_{\rm reg}$ of $M$, that is, $M$ minus its singular orbits, is equivariantly diffeomorphic to
\[
{M}_{\mathrm{reg}} \cong I \times P
\]
where $I\subset\mathbb{R}$ is an interval parametrizing the orbit space. The metric takes the form $g = dt^2 + g_t, $
where $g_t$ is a $G$-invariant metric on the orbit $P$. Nevertheless, in many symmetric situations, one has that the metric in the regular part can be written as
\[
g=dt+ f(t)^2 g_P ,
\]
for some fixed metric $g_P$ on $P$. We refer the interested reader to Section 7 from \cite{FePaTo23} for explicit examples, and to \cite{AlexBettiol} for details on cohomogeneity one actions.}
\newline

On the other hand, the study of submanifolds in warped products has also been a very active area in differential geometry (see, for instance, \cite{alias-dajczer,chen} and the references therein). {In this work, our interest lies in the study of curvatures properties of warped products and intersections of submanifolds, one of which is a warped product.}
\newline

Recently in \cite{melendez-hernandez}, Mel\'endez and Hern\'andez obtained an integral inequality of the Ricci curvature for the warped product $S^1\times_fN$, which gives a characterization of the simple product $S^1\times N$, where $N$ is a compact Riemannian manifold. In fact, they proved that in a warped product $\widebar{M}=S^1\times_fN$ the following estimate on $\widebar{M}$ holds:
\[
\int_{\widebar{M}} \Ric(\partial_t,\partial_t)\,\mathrm{d}{\widebar{M}^n}\geq0.
\]
%
%
Moreover, equality holds if and only if $\widebar{M}$ is simply a Riemannian product.

Following the main idea of the proof (see Theorem 4 in \cite{melendez-hernandez}), in our first result, we establish the next integral inequality for a more general warped product $M\times_fN$.

\begin{theorem} \label{theo:integral-inequality}
   Let $(M,g_M)$ and $(N,g_N)$ be Riemannian manifolds with $\dim M=m$, and assume that $M$ is Ricci flat. Let $\widebar{M}=M\times_f N$ be a compact warped product with warping function $f$. Then
    \begin{equation}
    \label{eqn:integral-inequality}
        \sum_{k=1}^m \int_{\widebar{M}} \Ric(E_k,E_k)\,\mathrm{d}{\widebar{M}}\geq0,
    \end{equation}
    where $\Ric$ denotes the Ricci curvature of $\widebar{M}$ and $\{E_1,\dots,E_m\}$ is a frame on $M$. Moreover, equality holds if and only if $\widebar{M}$ is a Riemannian product.
\end{theorem}
\smallskip

{The compactness of $M$ is essential.  To see this, consider the case where $M=\mathbb{R}$, then $M$ is automatically Ricci flat. For warped products, the Ricci curvature and the volume form satisfies}
\[
\Ric_{\widebar{M}}(\partial_t,\partial_t)
=
-(n-1)\frac{f''}{f}\partial_t, \quad \mathrm{d}{\widebar{M}^n} = f^{n-1}\,dt\,dN.
\]
Hence by integration by parts,
\[
\int_{\widebar{M}}\Ric(\partial_t,\partial_t)\,\mathrm{d}{\widebar{M}}
=
-(n-1){\rm Vol}(N)
\int_{\mathbb{R}} f''f^{n-2}\,dt
\]
which in general may neither converge nor be non-negative. Nevertheless, we may deduce an immediate corollary for cohomogeneity one manifolds.
\newline

If $(W, g)$ is a cohomogeneity one Riemannian manifold, by a Lie group $G$, it is well known that the orbit space is (diffeomorphic to) one of the following: i) $W/G=\mathbb{R}$, ii) $W/G=S^1$, iii) $W/G=[0, \infty)$, iv) $W/G=I$, for $I$ a closed interval. In cases i) and ii), all the orbits are principal; in cases i) and iii) $W$ is noncompact. If $\pi: W\to W/G$ denotes the quotient map, in case iv) there are exactly two singular orbits corresponding to $\pi^{-1}(-1)$ and $\pi^{-1}(1)$. In case ii), $W$ necessarily has infinite fundamental group.
\medskip

\begin{corollary}
    Let $(W, g)$ be a compact cohomogeneity one Riemannian manifold with orbit space equals $M=S^1$ or $M=I$, with $I$ a closed interval. Then 
    \begin{gather*}
        \int_{W_{reg}} \Ric(\partial_t,\partial_t)\,\mathrm{d}{W}\geq0,
    \end{gather*}
  where $\Ric$ denotes the Ricci curvature of $W$. {Moreover, equality holds if and only if the principal orbits are totally geodesic.}
\end{corollary}

\medskip

In order to illustrate the application of the theorem we will provide some explicit examples.

\begin{example}
    Consider the warped product $$\widebar{M}^n=T^{n-k}\times_fS^k$$ of the $(n-k)$-dimensional flat torus $T^{n-k}=S^1 \times \cdots \times S^1$ and the standard $k$-sphere $S^k$. Since Ricci tensor of $T^{n-k}$ is identically zero, we have
      \begin{gather*}
        \sum_{k=1}^{n-k}\int_{\widebar{M}} \Ric(\partial_{t_k},\partial_{t_k})\,\mathrm{d}{\widebar{M}}\geq0,
    \end{gather*}
 where $\partial_{t_k}\in TS^1$.
In particular, if $k=n-1$ we obtain  
\begin{gather*}        \widebar{M}=S^1\times_fS^{n-1} 
\end{gather*}
It follows from Theorem \ref{theo:integral-inequality} that
    \begin{gather*}
        \int_{\widebar{M}} \Ric(\partial_t,\partial_t)\,\mathrm{d}{\widebar{M}}\geq0,
    \end{gather*}
    where $\partial_t\in TS^1$. The equality holds if and only if $\widebar{M}$ is isometric to a Clifford hypersurface $S^1\times S^{n-1}$ (see Theorem 4 in \cite{melendez-hernandez} and Theorem 4.1 in \cite{min-seo}).  
\end{example}
\medskip

\begin{example}
A (complex) {\em $K3$ surface}  is simply connected compact complex manifold of dimension 2 with a nowhere-vanishing holomorphic 2-form. It is well-known that all $K3$ surfaces are diffeomorphic to a quartic $X\subset \mathbb{CP}^3$ \cite[Theorem 7.1.1]{Huybrechts}. Moreover, there exists a Kähler metric of vanishing Ricci curvature, due to the Calabi-Yau Theorem (see \cite[Theorem 9.4.11]{Huybrechts}). They $K3$ surfaces are the only compact simply connected Kähler manifolds of dimension $2$ that are Ricci-flat, and they are fundamental in several constructions of higher dimension compact Ricci-flat manifolds. 
\newline

An explicit example of a $K3$ surface is the {\em Fermat quartic}:
\[
\{ [x_0, x_1, x_2, x_3] \in \mathbb{CP}^3 \;\;|\;\;
x_0^4+x_1^4+x_2^4+x_3^4=0
\}.
\]

Therefore, the warped product of a $K3$ surface by any compact Riemannian manifold $N$ satisfies the hypothesis of  Theorem \ref{theo:integral-inequality}.

\end{example}
\medskip

Recall that the {\em holonomy group} of a Riemannian manifold at a point $p\in M$ is the group of linear transformations of the tangent space $T_pM$ obtained by parallel transporting tangent vectors along all piecewise smooth loops based at $p$, that is:
\[
\mathrm{Hol}_p(M,g)
=\left\{ P_\gamma : T_pM \to T_pM \ \middle|\ 
\gamma \text{ is a loop based at } p \right\},
\]
where $P_\gamma$ is the parallel transport map along $\gamma$. Geometrically, the holonomy group measures how tangent vectors rotate when along closed curves via the parallel transport.
\smallskip

\begin{example}
All the possible holonomy groups of simply connected manifolds are classified, due to  M. Berger (see for instance \cite[Theorem 3.4.1]{Joyce}). It is well-known that compact Riemannian manifolds whose holonomy group is contained in $ SU(n)$, $Sp(n)$, $G_2$ and $Spin(7)$ are Ricci-flat. See Propositions 6.1.1, 7.1.2, 10.1.5, and  10.5.5 from \cite{Joyce}, respectively. Therefore all of them can be taken as a compact manifold $N$, in order to produce warped products for Theorem \ref{theo:integral-inequality}. Some examples are the following.
\smallskip

D. Joyce constructed examples of compact $7$-dimensional Riemannian manifolds that are simply-connected and torsion-free with holonomy contained in $G_2$, by taking the flat 7-torus $T^7$, and taking the quotient by a finite group. Joyce also provide examples of 68 distinct, compact $7$-manifolds with holonomy $G_2$; and 95 distinct compact $8$-examples with holonomy Spin$(7)$. See \cite{Joyce} and the references therein. 
\smallskip

On the other hand, A. Kovalev obtained many compact Ricci-flat manifolds from a twisted connected sum construction of compact manifolds with holonomy in $G_2$ by two different methods. One method uses resolutions of singularities of appropriately chosen 7-dimensional orbifolds. Another method uses the gluing of two asymptotically cylindrical pieces and requires a certain matching condition on an embedded $K3$ surface.
\end{example}
\bigskip

%

%
%
%
%

Now we will discuss some properties that follow from considering certain intersections of a warped product with a submanifold in the same underlying ambient manifold. 
\newline

Consider the warped product
\[
M^n=\mathbb{R}\times_f P^{n-1},
\]
where $P$ is a Riemannian manifold and the metric is $dt^2+f(t)^2g_P$. 
For each $t\in\mathbb{R}$, the {\em slice} $\Sigma_t=\{t\}\times P$ is a totally umbilical hypersurface with constant mean curvature $\mathcal{H}(t)=\frac{f'(t)}{f(t)}$.
\begin{definition}
Let $\Sigma$ be a hypersurface in $M^n=\mathbb{R}\times_fP^{n-1}$. 
The \textit{height function} $h \colon \Sigma \to \mathbb{R}$ of  $\Sigma$ is defined by $h(p)=\pi_\mathbb{R}(p)$, $p\in \Sigma$, where $\pi_\mathbb{R} \colon M \to \mathbb{R}$ is the projection onto the first factor.    
\end{definition}
\medskip

\begin{definition}
    Let $M$ and $N$ be submanifolds of a Riemannian manifold $\widebar{M}$. We say that $M$ and $N$ are {\em normal submanifolds} in $\widebar{M}$ if
    \begin{enumerate}
        \item[\bfseries(a)] $M\cap N$ is a submanifold of $\widebar{M}$.
        \item[\bfseries(b)] $T_pM\cap T_p^\perp(M\cap N)\subset T_p^\perp N$ for all $p\in M\cap N$.
    \end{enumerate}    
\end{definition}
\medskip

\begin{remark}
\label{remark:normal}
    Condition (b) allows us to interchange the roles of $M$ and $N$, meaning that (b) is equivalent to $$T_pN\cap T_p^\perp(M\cap N)\subset T_p^\perp M$$ for all $p\in M\cap N$. In addition, it is not difficult to verify that if $M$ and $N$ are hypersurfaces, where $\eta$ and $\xi$ are the respective unit normal vector fields, then normality is equivalent to have $\phi=\measuredangle(\eta,\xi)=\frac{\pi}{2}$ along $M\cap N$. In this case $M$ and $N$ intersect transversally, thus $M\cap N$ is a submanifold of $\widebar{M}$ of codimension $2$.
\end{remark}


{See Section 4 in \cite{melendez-rodriguez} for more details on normal submanifolds.}
\newline

Let $M=\mathbb{R}\times_f P^{n-1}$. We now consider an isometric immersion $F: M \to \widebar{M}^{n+1}$ 
into a Riemannian manifold $\widebar{M}^{n+1}$. If   $N$ is a totally geodesic hypersurface of $\widebar{M}$, the next result gives us sufficient conditions for $M\cap N$ to be a totally geodesic slice of $M$, provided that $M$ and $N$ are normal hypersurfaces in $\widebar{M}$.


\begin{theorem} \label{theo:normality}
    Let $\widebar{M}^{n+1}$ be a Riemannian manifold, $N$ a totally geodesic hypersurface of $\widebar{M}$ and $M^n=\mathbb{R}\times_fP^{n-1}$ a warped product hypersurface in $\widebar{M}$ with $\mathcal{H}'(t)\geq0$, where $\mathcal{H}(t)=f'(t)/f(t)$. Suppose that $M$ and $N$ are normal hypersurfaces such that $\Sigma^{n-1}=M\cap N$ is a complete parabolic submanifold of $\widebar{M}$ with Ricci curvature bounded from below and bounded height function $h \colon \Sigma \to \mathbb{R}$. Then $\Sigma$ is a totally geodesic slice of $M$.
\end{theorem}

\begin{remark}
The condition $\mathcal{H}'(t)\geq 0$ is equivalent to the inequality:
\[
ff''-(f')^2\geq 0.
\]
This is not possible if
the Ricci curvature is positive in the direction $\partial_t$, since $
\Ric_M\left(\partial_t, \partial_t\right)=-\ \frac{f''}{f}\partial_t$.
\end{remark}
\medskip

The proof of Theorem \ref{theo:normality} is an application of the Omori–Yau maximum principle to $h$. The next corollary is a more practical version of the above theorem.

\begin{corollary} \label{coro:normality}
    Let $\widebar{M}^{n+1}$ be a Riemannian manifold, $N$ a totally geodesic hypersurface of $\widebar{M}$ and $M^n=\mathbb{R}\times_fP^{n-1}$ a warped product hypersurface in $\widebar{M}$ with $\mathcal{H}'(t)\geq0$. Suppose that $M$ and $N$ are normal hypersurfaces such that $\Sigma^{n-1}=M\cap N$ is compact. Then $\Sigma$ is a totally geodesic slice of $M$ and $P$ is compact.
\end{corollary}

\begin{example} 
    Consider a rotation hypersurface $M^n$ in the Euclidean space $\mathbb{R}^{n+1}$. We denote by $(x_1,\dots,x_{n+1})$ the coordinates in $\mathbb{R}^{n+1}$, and  parametrize $M$ by
    \begin{gather*}
        \varphi(t,s_1,\dots,s_{n-1})=\bigl(t,f(t)\Phi(s_1,\dots,s_{n-1})\bigr),
    \end{gather*}
    where $(t,f(t))$ is the profile curve of $M$, with $f(t)>0$ for all $t$, and $\Phi$ is a parametrization of the unit sphere $\mathbb{S}^{n-1}$.

    Observe that $M$ has the warped product metric
\begin{gather*}
\metric{\,}{}_{M}=dt^2+f(t)^2 g_{\mathbb{S}^{n-1}}
\end{gather*}
where $g_{\mathbb{S}^{n-1}}$ denotes the standard round metric of the sphere $\mathbb{S}^{n-1}$.

    If we assume that $\mathcal{H}'(t)\geq0$, then, as a direct application of Corollary \ref{coro:normality}, we see that the parallels
    \begin{gather*}
        \Sigma^{n-1}=M\cap N=\{t_0\}\times f(t_0)\,\mathbb{S}^{n-1}
    \end{gather*}
    are totally geodesic in $M$ if the intersection of $M$ with the horizontal hyperplane
    \begin{gather*}
        N=\bigl\{(x_1,\dots,x_{n+1})\in\mathbb{R}^{n+1} : x_1=t_0\bigr\}
    \end{gather*}
    is normal, which only happens when $f'(t_0)=0$.

%
%
\end{example}

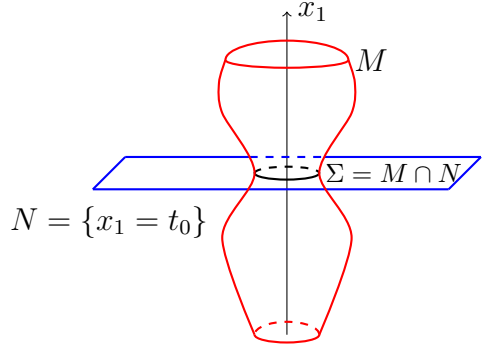
\begin{figure}[h]
        \centering
        \begin{tikzpicture}[scale=0.43]

            \draw[->] (0,-4 ) --  (0,6) node[right]{$x_1$};

            \draw[black,thick] (-1,1) arc (180:360:1 and .2);
            \draw[black,thick,dashed] (-1,1) arc (180:0:1 and .2);

         \draw[-,thick, blue] 
         (6,1.5) -- (1.15,1.5) node[right]{};
         \draw[-,thick, blue] 
         (-5,1.5) -- (-1.15,1.5) node[right]{};
        \draw[-,thick, blue] 
         (-6,0.5) -- (5,0.5) node[right]{};
         \draw[-,dashed,thick, blue] 
         (-1.15,1.5) -- (1.15,1.5) node[right]{};
         \draw[-,thick, blue] 
         (6,1.5) -- (5,0.5)
         node[right]{};
         \draw[-,thick, blue] 
         (-6,0.5) -- (-5,1.5)
         node[right]{};
        
          \node at (2.6,4.5) {$M$};
          \node at (3.3,1) {{\footnotesize $\Sigma=M\cap N$}};
           \node at (-5.5,-0.5) {$N=\{x_1=t_0\}$};
          
        
            \draw [red,thick] plot[smooth] coordinates {(-1.9,4.5) (-2.1,3.8) (-2,2.86) (-1,1) (-2,-1) (-1,-4)};
            
            \draw [red,thick] plot [smooth] coordinates {(1.9,4.5) (2.1,3.8) (2,2.86) (1,1) (2,-1) (1,-4)};
        
            \draw[dashed,red,thick] (1,-4) arc (0:180:1 and 0.4);
            \draw[red,thick] (-1,-4) arc (180:360:1 and .24);
        
        
            \draw[red,thick] (1.9,4.5) arc (0:180:1.9 and .6);
            \draw[red,thick] (-1.9,4.5) arc (180:360:1.9 and .24);
        
        
        
        \end{tikzpicture}
        \caption{A rotation hypersurface $M$ normal to  the parallel plane $x_1=t_0$}
        \label{fig:lady}
\end{figure}

Let $(M^{n+1},g)$ be a warped product of the form $M=I\times_f P$, where $P$ is a compact Riemannian manifold. In the following results, given a smooth embedded hypersurface $N\subset M$, we study the geometry of its intersection with the slices $\{t\}\times P$. Under suitable assumptions ensuring that these intersections are smooth and bound compact domains in $N$, we derive estimates relating the eigenvalues of the Laplacian on $P$ with geometric quantities of the hypersurface and the warping structure. The proofs rely on harmonic extensions and an application of Reilly's formula on the induced domains in $N$. 
\newline

\begin{theorem} \label{theo:firsteigen}
Let $(\widebar{M}^{n+1},g)$ be a Riemannian manifold and let 
\[
M = I \times_{f} P,\qquad g_M = dt^2 + f(t)^2 g_P,
\]
a warped product with $P$ a codimension $2$ compact submanifold of $\widebar{M}$. Let $N\subseteq \widebar{M}$ be a submanifold such that $\Sigma:=N\cap M\neq \emptyset$ is a $(n-1)$ submanifold of $M$. Then the first eigenvalue of the Laplacian of $N$ satisfies:
 \begin{align}
 \label{FirstLap1}
     \lambda_1(N)\leq \frac{2\int_{\Sigma}\norm{A_P(T)}^2+2\int_{\Sigma}\mathcal{H}^2(t)\cos^2(\theta)\sin^2(\theta)}{\int_{\Sigma}\cos^2(\theta)}
\end{align}    
where $\mathcal{H}(t)=f'(t)/f(t)$, $T$ is the tangential projection of $\partial_t$ over the tangent space of $P$, $A_P$ is the shape operator of $P$, and $\theta$ is the angle between $N$ and the normal slice.
\end{theorem}
\medskip

\begin{example}
Let $(\widebar{M}^{n+1}, g)$ be a compact cohomogeneity one manifold, under the action of a Lie group such that $\widebar{M}/G$ is a closed interval $I$ for which the metric in the regular part decompose as a warped product $M = I \times_f P$, where the principal orbit of the action. The vector field $\partial_t$ is the unit vector field orthogonal to the orbits. The mean curvature of the orbits is precisely $\mathcal H(t)=\frac{f'(t)}{f(t)}$. Therefore we may take the principal orbit as the submanifold $N$ and apply the Theorem \ref{theo:firsteigen}.
\newline

Besides the standard examples such as the sphere or the euclidean space, we may consider the manifold $\widebar{M}=\mathbb{CP}^2\#\overline{\mathbb{CP}^2}$. It is a cohomogeneity one manifold by means of the unitary group U$(2)$. There are at least two $U(2)$-invariant metrics, for instance, the Page metric \cite{PAGE1978} or the Koiso-Cao soliton \cite{Cao1}. The principal orbits are diffeomorphic to the sphere $S^3$ and the singular orbits are both diffeomorphic to $S^2$. In this case, we may take $N$ as the principal orbit, $N=S^3$. Therefore, with the repective metrics, the first eigenvalue of the Laplacian satisfies \eqref{FirstLap1}.
\end{example}
\medskip

\begin{theorem}
\label{theo: Reilly}
Let $(\widebar{M}^{n+1},g)$ be a Riemannian manifold and let 
\[
M^n = I \times_{f} P,\qquad g_M = dt^2 + f(t)^2 g_P,
\]
a warped product submanifold of $\widebar{M}$ with $P$ compact. Assume that $\mathcal{H}(t)>0$, where $\mathcal{H}(t)=f'(t)/f(t)$. Let $N\subseteq \widebar{M}$ be a submanifold such that $\Sigma:=N\cap M\neq \emptyset$ is a $(n-1)$ submanifold of $M$.  For a fixed regular value $t\in I$ of $f$ set:
\[
K_t := N \cap (\{t\}\times P),
\]
and assume that $K_t$ is a smooth, closed, connected submanifold which bounds a compact domain 
$\Omega_t\subset N$, and that the projection 
$\pi : \{t\}\times P \to P$ restricts to a diffeomorphism 
$\pi|_{K_t} : K_t \to P$ with the induced metric on $K_t$ equal to 
$f^2 g_P$.
\newline

Then the following spectral estimate holds for eigenvalues $\lambda$ of the Laplacian of $P$:

\[
\lambda^2
\ge
\frac{
\displaystyle
\int_{\Omega_t}\Ric_N(\nabla_N u, \nabla_Nº u)\,dV
+
\int_{K_t}\mathrm{II}_t(\nabla_t f_0,\nabla_t f_0)\,dA
}{
\displaystyle
\int_{K_t}\frac{f_0^2}{f^4 H_t}\,dA
}
\]
where $f_0$ is an eigenfunction of $\Delta_{K_t}$ with eigenvalue $\frac{\lambda}{f}$ and $\Ric_N$ is the  Ricci curvature of $N$.
\end{theorem}
\medskip

\begin{example}
Consider the following Riemannian manifold:
\[
\widebar{M}^{n+1} = (r_0,\infty)\times \mathbb{S}^{n}, \qquad 
g = \frac{dr^2}{1 - \frac{2m}{r^{n-1}}} + r^2 g_{\mathbb{S}^{n}},
\qquad m>0,
\]
which is a Schwarzschild-type metric. Introducing a new radial parameter $t$ defined by $
dt = \frac{dr}{\sqrt{1 - \frac{2m}{r^{n-1}}}}, $
we may take the warped product :
\[
M = I \times_f \mathbb{S}^{n-1}, \qquad g = dt^2 + f(t)^2 g_{\mathbb{S}^{n-1}}, \quad f(t)=r(t).
\]

Since ${
f'(t) = \sqrt{1 - \frac{2m}{f(t)^{n-1}}} > 0,}$ it follows that $
\mathcal{H}(t) = \frac{f'(t)}{f(t)} > 0.$

The isometric action of $SO(n+1)$ on $\mathbb{S}^{n}$ induces a cohomogeneity one structure on $\widebar{M}$, whose principal orbits are the hypersurfaces $\{t\}\times \mathbb{S}^{n-1}$.
\medskip

Let $N \subset \widebar{M}$ be the submanifold $N = (r_0,\infty)\times \mathbb{S}^{n-1}$, with the product metric, where $\mathbb{S}^{n-1} \subset \mathbb{S}^n$ is a totally geodesic equator. For each regular value $t\in I$, we have
\[
K_t = N \cap (\{t\}\times \mathbb{S}^{n-1}) = \{t\}\times \mathbb{S}^{n-1}.
\]
Then $K_t$ is a smooth, closed, connected hypersurface in $N$ which bounds a compact domain $\Omega_t \subset N$. The natural projection $\pi : K_t \to \mathbb{S}^{n-1}$
is a diffeomorphism, and the induced metric on $K_t$ is  $g_{K_t} = f(t)^2 g_{\mathbb{S}^{n-1}}$.
\newline

Therefore, all the hypotheses of Theorem \ref{theo: Reilly} are satisfied.
\end{example}

We organize the paper as follows. First we give some general preliminaries in Section \ref{sec:preliminaries}, where we address some properties of the intersection of submanifolds. Next we present the proofs of Theorems \ref{theo:integral-inequality}, \ref{theo:normality}, \ref{theo:firsteigen}  and \ref{theo: Reilly} in Sections \ref{sec:proof2} and \ref{sec:proof3}, respectively, where we give some notation and auxiliary results used in each proof. 

\section{Preliminaries on warped products and intersections} \label{sec:preliminaries}

Throughout the manuscript $\widebar{M}$ will always denote the ambient space, which is a Riemannian manifold with metric $\metric{\,}{}$ and Riemannian connection $\nabla$. $M$ and $N$ also will denote Riemannian manifolds with metrics $\metric{\,}{}_M$ and $\metric{\,}{}_N$, and Riemannian connections $\nabla^M$ and $\nabla^N$, respectively.

Let $f \colon M \to \mathbb{R}^+$ be a smooth function. The \textit{warped product} $\widebar{M}=M\times_fN$ is the product manifold $M\times N$ endowed with the warped metric
\begin{gather*}
    \metric{X}{Y}=\metric{d\pi_M(X)}{d\pi_M(Y)}_M+(f\circ\pi_M)^2\metric{d\pi_N(X)}{d\pi_N(Y)}_N,
\end{gather*}
where $\pi_M$ and $\pi_N$ are the projections of $\widebar{M}$ onto the corresponding factor. In a compact way we write $$\metric{\,}{}=\metric{\,}{}_M+f^2\metric{\,}{}_N.$$ The function $f$ is called the \textit{warping function} of $\widebar{M}$. Notice that if $f$ is constant, then $M\times_fN$ is the Riemannian product $M\times N$ where $N$ has the metric $f^2\metric{\,}{}_N$.

Let $f \colon \widebar{M} \to \mathbb{R}$ be a smooth function. We denote by $\grad\,f$ the \textit{gradient} of $f$, and by $\Hess\,f$ the Hessian of $f$, which are defined by
\begin{gather*}
    df(X)=\metric{\grad\,f}{X} \quad \text{and} \quad \Hess\,f(X,Y)=\metric{\nabla_X(\grad\,f)}{Y},
\end{gather*}
where $X,Y$ are vector fields in $\widebar{M}$. In addition, the \textit{Laplacian} of $f$ is considered with the sign convention $$\Delta f=\tr(\Hess\,f),$$ where $\tr$ denotes the trace of a linear operator.

We use a similar notation, but with a superscript, for the above differential operators when $f$ is a real-valued function defined in $M$ or $N$. For example, we write $\grad^Mf$, $\Hess^Mf$ and $\Delta^Mf$.

\begin{lemma} \label{lemma:oneill}
    Let $M$ and $N$ be Riemannian manifolds, with $n=\dim(N)>1$, and let $f \colon M \to \mathbb{R}^+$ be a smooth function. Consider the warped product $\widebar{M}=M\times_fN$. Let $X,Y$ be horizontal vector fields, and let $V$ be a vertical vector field. Then
    \begin{enumerate}
        \item[\bfseries(1)] $\Ric(X,Y)=\Ric^M(X,Y)-\medmath{\frac{n}{f}}\,\Hess^Mf(X,Y)$.
        \item[\bfseries(2)] $\nabla_XV=\nabla_VX=\medmath{\frac{X(f)}{f}}V$.
    \end{enumerate}
\end{lemma}

We also need the next technical lemma contained in Proposition 2.3 of \cite{dobarro-unal} (see also Lemma 4 of \cite{melendez-hernandez}).

\begin{lemma} \label{lemma:dobarro}
    Let $M$ and $N$ be Riemannian manifolds, with $n=\dim(N)$, and let $f \colon M \to \mathbb{R}^+$ be a smooth function. Consider the warped product $\widebar{M}=M\times_fN$ and $u\in C^\infty(\widebar{M})$. Then
    \begin{gather*} \label{eq:lema-laplaciano}
        \Delta u=\Delta^Mu+\frac{n}{f}\metric{\grad^Mf}{\grad^Mu}_M+\frac{1}{f^2}\Delta^Nu.
    \end{gather*}
\end{lemma}
\medskip

In the sequel, we will need tome results related to a property of the normal intersection of two hypersurfaces in a Riemannian ambient. This auxiliary result can be seen as a consequence of a general pattern concerning the intersection of submanifolds of arbitrary codimension. In fact it will be relevant in the proof of Theorem \ref{theo:normality}.
\medskip


\begin{proposition} \label{prop:intersection}
    Let $M$ and $N$ be submanifolds of a Riemannian manifold $\widebar{M}$ such that $\Sigma=M\cap N$ is a submanifold of $\widebar{M}$. If $N$ is totally umbilical in $\widebar{M}$, then
    \begin{gather*}
        \mathbf{H}_\Sigma^M=(\mathbf{H}_\Sigma^N)^\top+(\mathbf{H}_N^{\widebar{M}})^\top \quad \text{along $\Sigma$},
    \end{gather*}
    where $(\cdot)^\top$ is the projection over $TM$.
\end{proposition}

\begin{proof}
    Recall that for all $X,Y\in\mathfrak{X}(\Sigma)$, we have the decomposition
    \begin{gather*}
        B_\Sigma^{\widebar{M}}(X,Y)= B_\Sigma^M(X,Y)+B_M^{\widebar{M}}(X,Y)=B_\Sigma^N(X,Y)+B_N^{\widebar{M}}(X,Y).
    \end{gather*}
    Therefore, if $k=\dim(\Sigma)$ and $\{E_1,\dots,E_k\}$ is a local orthonormal frame in $\Sigma$, we deduce from de umbilicity of $N$ in $\widebar{M}$ that
    \begin{gather*}
        \begin{split}
            k\mathbf{H}_\Sigma^M&=\sum_{i=1}^k B_\Sigma^M(E_i,E_i) \\
            &=\sum_{i=1}^k B_\Sigma^N(E_i,E_i)+\sum_{i=1}^k B_N^{\widebar{M}}(E_i,E_i)-\sum_{i=1}^k B_M^{\widebar{M}}(E_i,E_i) \\
            &=k\mathbf{H}_\Sigma^N+k\mathbf{H}_N^{\widebar{M}}-\sum_{i=1}^k B_M^{\widebar{M}}(E_i,E_i),
        \end{split}
    \end{gather*}
    and by projecting the last equation to $TM$ we get the desired formula.
\end{proof}

When $M$ and $N$ are hypersurfaces and do not intersect ``tangentially'', we obtain:

\begin{corollary} \label{coro:intersection}
    Let $M$ and $N$ be oriented hypersurfaces of $\widebar{M}$, where $\eta$ and $\xi$ are their respective unit normal 
    vectors, and suppose that $M \cap N \neq \emptyset$. If $N$ is totally umbilical in $\widebar{M}$ and the angle $\phi=\angle(\eta,\xi)\in(0,\pi)$, then
    \begin{gather*}
        \mathbf{H}_\Sigma^M=\left(H_\Sigma^N\cos(\phi)\pm H_N^{\widebar{M}}\sin(\phi)\right)\eta_* \quad \text{along $\Sigma=M\cap N$},
    \end{gather*}
    where $\eta_*$ is the unit normal vector of $\Sigma$ respect to $M$, and $H_\Sigma^N$ and $H_N^{\widebar{M}}$ are the scalar mean curvatures of $\Sigma\subset N$ and $N\subset\widebar{M}$, respectively. In particular, when $M$ and $N$ are normal hypersurfaces in $\widebar{M}$, we have
    \begin{gather*}
        \mathbf{H}_\Sigma^M=\mathbf{H}_N^{\widebar{M}}= H_N^{\widebar{M}}\xi
    \end{gather*}
    and $\Sigma$ is totally umbilical in $M$.
\end{corollary}

\begin{proof}
    As $M$ and $N$ are hypersurfaces and $\phi\in(0,\pi)$, necessarily $M$ and $N$ intersect transversally, so $\Sigma$ is a submanifold of $\widebar{M}$ of codimension $2$ (see Remark \ref{remark:normal}).
    
    Let $\eta_*$ and $\xi_*$ be the unit normal vectors of $\Sigma$ respect to $\eta$ and $\xi$, respectively, meaning that $\{\eta_*,\eta\}$ and $\{\xi_*,\xi\}$ are positively oriented orthonormal bases of the plane $T_p^\perp\Sigma$ for all $p\in\Sigma$. Consequently the frame $\{\xi_*,\xi\}$ is obtained by rotating the frame $\{\eta_*,\eta\}$ an angle $\phi$, and depending on the position of the frames, the rotation can be clockwise or counterclockwise. Therefore we can write
    \begin{gather*}
        \xi_*=\cos(\phi)\eta_*\pm\sin(\phi)\eta, \quad \xi=\mp\sin(\phi)\eta_*+\cos(\phi)\eta,
    \end{gather*}
    which implies
    \begin{gather*}
        (\mathbf{H}_\Sigma^N)^\top=(H_\Sigma^N\xi_*)^\top=H_\Sigma^N\cos(\phi)\eta_*, \quad (\mathbf{H}_N^{\widebar{M}})^\top=(H_N^{\widebar{M}}\xi)^\top=\mp H_N^{\widebar{M}}\sin(\phi)\eta_*.
    \end{gather*}
    Therefore, from Proposition \ref{prop:intersection} we obtain
    \begin{gather*}
        \mathbf{H}_\Sigma^M=(\mathbf{H}_\Sigma^N)^\top+(\mathbf{H}_N^{\widebar{M}})^\top=\left(H_\Sigma^N\cos(\phi)\mp H_N^{\widebar{M}}\sin(\phi)\right)\eta_*.
    \end{gather*}
    If $M$ and $N$ are normal hypersurfaces in $\widebar{M}$ ($\phi\equiv\frac{\pi}{2}$), 
    then 
    \begin{gather*}
        \mathbf{H}_\Sigma^M=(\mathbf{H}_N^{\widebar{M}})^\top=\mp H_N^{\widebar{M}}\eta_*
        =H_N^{\widebar{M}}\xi.
    \end{gather*}\end{proof}

\section{Proofs of the main results}

\subsection{Proof of Theorem \ref{theo:integral-inequality}} \label{sec:proof2}

We denote by $\Ric$ and $\Ric^M$ the Ricci curvatures of $\widebar{M}$ and $M$, respectively. For the proof of Theorem \ref{theo:integral-inequality} we need a well-known fact about warped products, and given that it shares the same context with another fact that will be used in the proof of Theorem \ref{theo:normality}, we present both in the following lemma (see Chapter 7 in \cite{oneill}).

\begin{proof}[\bfseries Proof of Theorem \ref{theo:integral-inequality}]
    Let $\{E_1,\dots,E_m\}$ be an orthonormal frame in $M$. Since $M$ is Ricci flat, Lemma \ref{lemma:oneill} implies
    \begin{gather} \label{eq:ricci}
        \sum_{k=1}^m \Ric(E_k,E_k)=-\frac{n}{f}\sum_{k=1}^m \Hess^Mf(E_k,E_k)=-\frac{n}{f}\Delta^Mf.
    \end{gather}
    If we set $u=\ln f$, we obtain
    \begin{gather} \label{eq:laplacian1}
        \begin{split}
            \Delta^Mu&=\sum_{k=1}^m \metric{\nabla^M_{E_k}(\grad^Mu)}{E_k}_M \\
            &=\sum_{k=1}^m \metric{\nabla^M_{E_k}\left(\medmath{\frac{\grad^Mf}{f}}\right)}{E_k}_M=\frac{1}{f}\Delta^Mf-\frac{1}{f^2}\norm{\grad^Mf}_M^2.
        \end{split}
    \end{gather}
    Now, by using Lemma \ref{lemma:dobarro},
    \begin{gather} \label{eq:laplacian2}
        \Delta u=\Delta^Mu+\frac{n}{f}\metric{\grad^Mf}{\grad^Mu}_M=\Delta^Mu+\frac{n}{f^2}\norm{\grad^Mf}_M^2.
    \end{gather}

    If we substitute \eqref{eq:laplacian1} in \eqref{eq:laplacian2}, we find that
    \begin{gather*}
        \Delta u=\frac{\Delta^Mf}{f}+(n-1)\left(\frac{\norm{\grad^Mf}_M}{f}\right)^2=\frac{\Delta^Mf}{f}+(n-1)\norm{\mathcal{H}}_M^2,
    \end{gather*}
    where $\mathcal{H}=\medmath{\frac{1}{f}}\,\grad^Mf$. It follows from \eqref{eq:ricci} that
    \begin{gather} \label{eq:laplacian3}
        n\Delta u=-\sum_{k=1}^m \Ric(E_k,E_k)+n(n-1)\norm{\mathcal{H}}_M^2.
    \end{gather}
    By the compactness of $\widebar{M}$, we can integrate both sides of \eqref{eq:laplacian3} to obtain
    \begin{gather*}
        \sum_{k=1}^m \int_{\widebar{M}} \Ric(E_k,E_k)\,d{\widebar{M}}=n(n-1)\int_{\widebar{M}} \norm{\mathcal{H}}_M^2\,d{\widebar{M}}\geq0.
    \end{gather*}
    Observe that equality holds if and only if $\mathcal{H}=0$, which means that $f$ is constant.
\end{proof}

\subsection{Proof of Theorem \ref{theo:normality}} \label{sec:proof3}

We say that $M$ is \textit{parabolic} if the only subharmonic functions on $M$ which are bounded from above are the constant ones. Explicitly stated, this means that if $u\in C^2(M)$ is such that $\Delta^Mu\geq0$ and $\sup_Mu<\infty$, then $u$ must be constant. Being parabolic is equivalent to have that the only superharmonic functions on $M$ which are bounded from below are the constant ones, or explicitly, if $u\in C^2(M)$ satisfies $\Delta^Mu\leq0$ and $\inf_Mu>-\infty$, then $u$ must be constant.

It is well known that every compact manifold is parabolic. In particular, any sphere $S^n$ is parabolic. On the other hand, the Euclidean space $\mathbb{R}^n$ is parabolic if and only if $n=1,2$. To see that $\mathbb{R}^n$ is not parabolic for $n\geq3$, it is sufficient to give an explicit example of a positive non constant superharmonic function, like the map
\begin{gather*}
    u(x)=(1+\norm{x}^2)^{-\frac{n-2}{2}}.
\end{gather*}

For the proof of Theorem \ref{theo:normality} we use the following well-known principle due to H. Omori and S. T. Yau (see \cite{yau}).

\begin{theorem}
    Let $M$ be a complete Riemannian manifold whose Ricci curvature is bounded from below. Consider $u\in C^2(M)$ that is bounded from below on $M$. Then there exists a sequence $\{p_j\}$ in $M$ such that
    \begin{gather} \label{eq:omori-yau}
        \lim_{j\to\infty}u(p_j)=\inf_Mu, \quad \norm{\grad^Mu(p_j)}<\frac{1}{j}, \quad \Delta^Mu(p_j)>-\frac{1}{j}.
    \end{gather}
\end{theorem}
\medskip

Let $N$ be a submanifold of $\widebar{M}$. Let us denote by $B_N^{\widebar{M}}$ the second fundamental form of $N$ in $\widebar{M}$, this means that
\begin{gather*}
    \widebar{\nabla}_X Y=\nabla_X^NY+B_N^{\widebar{M}}(X,Y), \quad X,Y\in\mathfrak{X}(N).
\end{gather*}
If $n=\dim(N)$,  the mean curvature vector of $N$ in $\widebar{M}$ is given by
\begin{gather*}
\mathbf{H}_N^{\widebar{M}}=\frac{1}{n}\sum_{i=1}^n B_N^{\widebar{M}}(E_i,E_i),
\end{gather*}
where $\{E_1,\dots,E_n\}$ is a local orthonormal frame on $N$. A submanifold $N$  is said to be totally umbilical  if 
$$B_N^{\widebar{M}} (X,Y) = \langle X, Y \rangle_N \mathbf{H}_N^{\widebar{M}}$$
for every $X, Y \in \mathfrak{X} (N)$.
\newline

Let $P^{n-1}$ be a Riemannian manifold with metric $\metric{\,}{}_P$, and consider the warped product $M^n=\mathbb{R}\times_fP$ with its warped metric $\metric{\,}{}=dt^2+f^2\metric{\,}{}_P$. Let $\Sigma$ be a complete oriented hypersurface of $M$, where $\eta$ is the unit normal vector of $\Sigma$. Then the shape operator of $\Sigma$ respect to $\eta$ is given by
\begin{gather*}
    AX=-\nabla_X\eta, \quad \text{where $X\in\mathfrak{X}(M)$}.
\end{gather*}

We also need to compute the Laplacian of the height function $h \colon \Sigma \to \mathbb{R}$ (see Proposition 2.1 of \cite{alias-dajczer}). For the reader's convenience we include a detailed proof.

\begin{lemma}
    Let $\Sigma^{n-1}$ be an oriented hypersurface of $M^n=\mathbb{R}\times_fP$. Then
    \begin{gather} \label{eq:alias-dajczer}
        \Delta^\Sigma h=\mathcal{H}(h)\bigl(n-1-\norm{\grad^\Sigma h}^2\bigr)+(n-1)\metric{\partial_t}{\mathbf{H}_\Sigma^M},
    \end{gather}
    where $\mathbf{H}_\Sigma^M$ is the mean curvature vector of $\Sigma$ respect to $M$ and $\partial_t$ is the coordinate vector field of the first factor of $M$.
\end{lemma}

\begin{proof}
    First note that the gradient of the projection $\pi_1 \colon M \to \mathbb{R}$ is given by $\grad^M\pi_1=\partial_t$. Thus the gradient of the height function $h(p)=\pi_1(p)$ is given by
    \begin{gather} \label{eq:gradient-h}
        \grad^\Sigma h=(\grad^M\pi_1)^\top=\partial_t^\top=\partial_t-\metric{\partial_t}{\eta}\eta,
    \end{gather}
    where $\left(\cdot\right)^\top$ denotes the projection over $T\Sigma$. On the other hand, we know that each $X\in\mathfrak{X}(M)$ can be decomposed as $X=\metric{X}{\partial_t}\partial_t+V$, where $V=X-\metric{X}{\partial_t}\partial_t$. It follows from item (2) of Lemma \ref{lemma:oneill} that
    \begin{gather} \label{eq:partialt}
        \nabla_X^M\partial_t=\metric{X}{\partial_t}\cancelto{0}{\nabla_{\partial_t}\partial_t}+\nabla_V\partial_t=\mathcal{H}(t)V=\mathcal{H}(t)(X-\metric{X}{\partial_t}\partial_t),
    \end{gather}
    where $\mathcal{H}(t)=f'(t)/f(t)$. Using \eqref{eq:gradient-h} and \eqref{eq:partialt} we deduce that
    \begin{gather*}
        \nabla_X^M(\grad^\Sigma h)=\nabla_X(\partial_t-\metric{\partial_t}{\eta}\eta)=\mathcal{H}(t)(X-\metric{X}{\partial_t}\partial_t)-X\metric{\partial_t}{\eta}\eta-\metric{\partial_t}{\eta}\nabla_X\eta.
    \end{gather*}
    If we consider $X\in\mathfrak{X}(\Sigma)$ in the above equation, by projecting over $T\Sigma$ we get
    \begin{gather*}
        \begin{split}
            \nabla_X^\Sigma(\grad^\Sigma h)&=\bigl(\nabla_X^M(\grad^\Sigma h)\bigr)^\top \\
            &=\mathcal{H}(h)\bigl(X-\metric{X}{\partial_t}\partial_t^\top\bigr)-\metric{\partial_t}{\eta}(\nabla_X\eta)^\top \\
            &=\mathcal{H}(h)\bigl(X-\metric{X}{\grad^\Sigma h}\grad^\Sigma h\bigr)+\metric{\partial_t}{\eta}AX.
        \end{split}
    \end{gather*}
    Therefore, if we consider an orthonormal frame $\{E_i\}$ on $\Sigma$, we conclude that
    \begin{gather*}
        \begin{split}
            \Delta^\Sigma h&=\sum_{i=1}^{n-1} \metric{\nabla_{E_i}^\Sigma(\grad^\Sigma h)}{E_i} \\
            &=\sum_{i=1}^{n-1} \left(\mathcal{H}(h)\bigl(1-\metric{E_i}{\grad^\Sigma h}^2\bigr)+\metric{\partial_t}{\eta}\metric{AE_i}{E_i}\right) \\
            &=\mathcal{H}(h)\bigl(n-1-\norm{\grad^\Sigma h}^2\bigr)+(n-1)\metric{\partial_t}{\mathbf{H}_\Sigma^M}.
        \end{split}
    \end{gather*} 
\end{proof}

\begin{proof}[\bfseries Proof of Theorem \ref{theo:normality}]
    Let $M=\mathbb{R}\times_fP$ be a warped product hypersurface and $N$ a totally geodesic hypersurface of $\widebar{M}$, and assume that $M$ and $N$ are normal hypersurfaces. Then Corollary \ref{coro:intersection} 
    implies
    \begin{gather*}
        \mathbf{H}_{\Sigma}^M=\mathbf{H}_N^{\widebar{M}}=0 \quad \text{on $\Sigma=M \cap N$}.
    \end{gather*}
    Then equation \eqref{eq:alias-dajczer} yields
    \begin{gather} \label{eq:laplacian4}
        \Delta^\Sigma h=\mathcal{H}(h)\left(n-1-\norm{\grad^\Sigma h}^2\right).
    \end{gather}

    By the Omori-Yau maximum principle \eqref{eq:omori-yau} we know that there exists a sequence $\{p_j\}$ in $\Sigma$ such that
    \begin{gather*}
        \lim_{j\to\infty}h(p_j)=h_*, \quad \norm{\grad^\Sigma h(p_j)}^2<\frac{1}{j^2}, \quad \Delta^\Sigma h(p_j)>-\frac{1}{j},
    \end{gather*}
    where $h_*=\inf_\Sigma h$, and by the previous equation we obtain
    \begin{gather*}
        \mathcal{H}(h(p_j))(n-1-\norm{\grad^\Sigma h(p_j)}^2)>-\frac{1}{j},
    \end{gather*}
    which reduces to $\mathcal{H}(h_*)\geq0$ when $j\to\infty$. Since $\mathcal{H}'\geq0$, we get
    \begin{gather} \label{eq:H}
    \mathcal{H}(h)\geq0 \quad \text{on $\Sigma$}.
    \end{gather}

    Let $\eta$ be the unit normal vector of $\Sigma$ in $M$. As $\grad^\Sigma h=\partial_t-\metric{\partial_t}{\eta}\eta$, it follows
    \begin{gather*}
        \norm{\grad^\Sigma h}^2=1-\metric{\partial_t}{\eta}^2.
    \end{gather*}
    Therefore
    \begin{gather*} \label{eq:positive}
        (n-1)-\norm{\grad^\Sigma h}^2=(n-2)+\metric{\partial_t}{\eta}^2\geq0.
    \end{gather*}

    By combining \eqref{eq:laplacian4} and \eqref{eq:H}, we conclude $\Delta^\Sigma h\geq0$ on $\Sigma$, but given that $\sup_\Sigma h<\infty$, the parabolicity of $\Sigma$ implies that $h$ is constant with $\mathcal{H}(h)=0$, and consequently $\Sigma$ is a totally geodesic slice of $M$.
\end{proof}

\subsection{Proof of Theorem \ref{theo:firsteigen}} \label{sec:proof3}

Recall the Rayleigh quotient for the first eigenvalue of the Laplacian $\lambda_1(N)$, 
\[
\lambda_1(N)\leq \frac{\int_N \norm{\nabla \varphi}^2}{\int_N \varphi^2}.
\]
for any smooth function $\varphi$ on $N$. Let $\xi$ be unit normal of $N$ in $\widebar{M}$, and define a function smooth $\varphi$ on $N$ as $\cos(\theta)=\langle \xi, \partial_t\rangle$ in $\Sigma$ and extending it smoothly by zero outside $\Sigma$.
\newline

For $X$ a vector field on $\Sigma$, we differentiate:
\begin{align*}
    X\cos(\theta)=X\langle \xi, \partial_ t\rangle=&\langle \nabla_X\xi, \partial_t\rangle+\langle \xi, \nabla_X\partial_t\rangle\\
    =&-\langle A_P(X), \partial_ t\rangle+\langle \xi, \nabla_X\partial_t\rangle.
\end{align*}

On the other hand, we decompose: $ X=X_P+\langle X,\partial_t  \rangle \partial_t$, where $X_P$ is tangent to $P$. Since $X$ is tangent to $N$, then $\langle\xi, X \rangle=0$ and,
\[
\langle\xi, X_P \rangle=\langle \xi, X-\langle X, \partial_t\rangle\partial_t\rangle=-\langle X,\partial_t  \rangle\cos(\theta).
\]

Using that $\nabla_ X\partial_ t=\mathcal{H}(t)X$ and that $\nabla_{\partial_t}\partial_t=0$, we have: 
\[
X\cos(\theta)=-\langle A_P(X), \partial_t\rangle-\mathcal{H}(t)\langle X, \partial_t \rangle\cos(\theta).
\]

Denote by $T$ the tangential projection of $\partial_t$ to $T_xP$, that is:
\[
T:=\partial_t-\langle\partial_t, \xi \rangle\xi=\partial_t-\cos(\theta)\xi.
\]
Therefore $\norm{T}=\sin(\theta)$. Hence we obtain:
\[
\nabla^N\cos(\theta)=-A_P(T)-\mathcal{H}(t)\cos(\theta)T.
\]

It follows that: 
\[
\norm{\nabla^{N}\cos(\theta)}^2=\norm{A_P(T)}^2+2\mathcal{H}(t)\cos(\theta)\langle A_P(T), T \rangle+\mathcal{H}^2(t)\cos^2(\theta)\sin^2(\theta).\]
Thus:
\[
\int_N \norm{\nabla^{N}\cos(\theta)}^2\ dV
\leq2\int_P \norm{A_P(T)}^2+2\int_P\mathcal{H}^2(t)\cos^2(\theta)\sin^2(\theta),
\]
from which the result follows.



\subsection{Proof of Theorem \ref{theo: Reilly}} \label{sec:proof4}


Let $\psi$ be an eigenfunction on $P$ associated to the eigenvalue $\lambda$, that is:
\[
-\Delta_P\psi=\lambda\psi, \qquad \int_P\psi=0
\]

From the hypothesis of Theorem \ref{theo: Reilly}, via the diffeomorphism $\pi|_{K_t}:K_t\to P$ we define $f_0 := \psi\circ(\pi|_{K_t})^{-1}$ on $K_t$. Observe that since the metric at the slice $\{t\}\times g_P$ is $f^2\cdot g_P$, it implies that the Laplacian rescales $
	\Delta_{K_t}=f^{-2}\Delta_{P}$, 	and therefore we have the relation:
	\begin{align}
		\label{eqn: Eigen}
	-\,\Delta_{K_t} f_0 = \frac{\lambda}{ f^2}\, f_0.
	\end{align}
	
	Let $u\in C^\infty(\overline{\Omega}_t)$ be the solution to the following Dirichlet problem:
	\[
	\begin{cases}
    \Delta_N u= 0 & \mbox{in $\Omega_t,$}\\
	\quad \ u= f_0 & \mbox{in $K_t$}.
		\end{cases}
	\]
Therefore, by \eqref{eqn: Eigen} that on $K_t$:
	\[
	\Delta_{K_t}u = -\frac{\lambda}{ f(t)^2}\, f_0.
	\]
\smallskip
        
	We use the following form of Reilly's formula (see \cite[Lemma A.17]{bCLN}) on the compact manifold  $(\Omega_t,g|_N)$ with boundary $K_t$:
	{\small\[
	\int_{\Omega_t}\Big( (\Delta_N u)^2 - |\nabla_N^2 u|^2 - \Ric_N(\nabla_N u,\nabla_N u)\Big)\,dV
	=
	\int_{K_t}\Big( H_{t}\,u_\nu^2 - 2\,u_\nu\,\Delta_{t}u + \mathrm{II}_{t}(\nabla_{t}u,\nabla_{t}u)\Big)\ dA.
	\]}
    \smallskip
    
	Here $\nu$ is the outward unit normal to $\Omega_t$, $u_\nu$ is the directional derivative of $u$ in the direction of $\nu$; while $\nabla_t$, $H_{t}, \mathrm{II}_{t}$ and $\nabla_t$ are the Laplacian, the mean curvature of $K_t\subset N$, its second fundamental form and the Levi-Civita connection, respectively.
	\newline	
	
	Since $\Delta_N u=0$ in $\Omega_t$, the Reilly's formula becomes:
    \smallskip
    
	\[
	- \int_{\Omega_t}\Big( |\nabla_N^2 u|^2 + \Ric_N(\nabla_N u,\nabla_N u)\Big)\,dV
	=
	\int_{K_t}\Big( H_{t}\,u_\nu^2 - 2\,u_\nu\,\Delta_{t}u + \mathrm{II}_{t}(\nabla_{t}u,\nabla_{t}u)\Big).
	\]

	On the other hand, using \eqref{eqn: Eigen}, the boundary integral in the previous equals
	\[
	\int_{K_t}\Big( H_{t}\,u_\nu^2 + 2\frac{\lambda}{ f^2} f_0\,u_\nu + \mathrm{II}_{t}(\nabla_{t} f_0,\nabla_{t} f_0)\Big).
	\]
	
	Replacing this in the Reilly's formula, we obtain:
	\[
	- \int_{\Omega_t}\Big( |\nabla_N^2 u|^2 + \Ric_N(\nabla_N u,\nabla_N u)\Big)\,dV
	=
	\int_{K_t}\Big( H_t\,u_\nu^2 + 2\frac{\lambda}{ f^2} f_0\,u_\nu
	+ \mathrm{II}_{t}(\nabla_{t} f_0,\nabla_{t} 
    f_0)\Big)\,dA.	\]
\medskip

Then, since $H_t>0$, we complete the square in $u_\nu$
\[
H_t\,u_\nu^2 + 2\frac{\lambda}{f(t)^2} f_0\,u_\nu
=
H_t\left(u_\nu + \frac{\lambda}{f^2 H_t}f_0\right)^2
-
\frac{\lambda^2}{f^4 H_t} f_0^2 .
\]

Since $H_t>0$, the square term is nonnegative, which yields the estimate
\[
H_t\,u_\nu^2 + 2\frac{\lambda}{f^2} f_0\,u_\nu
\ge
-\frac{\lambda^2}{f^4 H_t} f_0^2 .
\]

Substituting this inequality into the boundary integral in Reilly's formula gives
\[
\int_{K_t}\Big( H_t\,u_\nu^2 + 2\frac{\lambda}{f^2} f_0\,u_\nu
+ \mathrm{II}_{t}(\nabla_{t} f_0,\nabla_{t} f_0)\Big)\,dA
\ge
\int_{K_t}\Big(
-\frac{\lambda^2}{f^4 H_t} f_0^2
+ \mathrm{II}_{t}(\nabla_{t} f_0,\nabla_{t} f_0)
\Big)\,dA .
\]

Dropping the nonnegative square on the right-hand side yields the inequality
	\[
	- \int_{\Omega_t}\Ric_N(\nabla_N u,\nabla_N u)\,dV
	\ge
	\int_{K_t}\Big( -\frac{\lambda^2}{f^4 H_t}f_0^2 + \mathrm{II}_t(\nabla_{t}f_0,\nabla_{t}f_0)\Big)\,dA.
	\]

 \medskip

It follows that:

\[
\int_{\Omega_t}\Ric_N(\nabla_N u,\nabla_N u)\,dV
	\ge
	\int_{K_t}\Big( -\frac{\lambda^2}{f^4 H_t}f_0^2 + \mathrm{II}_t(\nabla_{t}f_0,\nabla_{t}f_0)\Big)\,dA.
\]    
\medskip

Therefore, rearranging the inequality we obtain the desired bound.$\qed$.
\newline

Finally, we have the following estimate of $II_t$ for eigenfunctions of the Laplacian on $K_t$, in terms of the warping function $f$. 

\begin{proposition}
Let $(\widebar{M}^{n+1},g)$ be a Riemannian manifold and let 
\[
M^n = I \times_{f} P,\qquad g_M = dt^2 + f(t)^2 g_P,
\]
a warped product with $P$ compact. Keeping the notations from Theorem \ref{theo: Reilly}, we have:
\[
\int_{K_t}\mathrm{II}_t(\nabla_t f_0,\nabla_t f_0)\,dA
\ge \kappa_t\frac{\lambda}{f^2}
\int_{K_t}f_0^2\,dA .
\]
for some constant $\kappa_t$ at the slice at $t\in I$.
\end{proposition}
\begin{proof}
The Green's identity on a compact Riemannian manifold in this case reads: 
\[
\int_{\Omega_t}\langle \nabla_t u, \nabla_t v\rangle dV=-\int_{\Omega_t}v\Delta_t u+\int_{K_t}v\cdot u_{\nu}\ dA
\]

Let $u$ be the harmonic extension of $f_0$ in $\Omega_t$, as in the proof of Theorem \ref{theo: Reilly}. Since $u$ is harmonic in $\Omega_t$ we obtain: 
\[
\int_{\Omega_t} \norm{\nabla_N u}^2\,dV
=
\int_{K_t} f_0\,u_\nu\, dA .
\]

Moreover, using $-\Delta_{K_t}f_0=\frac{\lambda}{f(t)^2}f_0$, we compute
\[
\int_{K_t} \norm{\nabla_t f_0}^2\,dA
=
-\int_{K_t} f_0\,\Delta_{t}f_0\,dA
=
\frac{\lambda}{f^2}\int_{K_t} f_0^2\,dA .
\]

For each $t$, there exists a constant $\kappa_t\in\mathbb{R}$ such that $\mathrm{II}_t \ge \kappa_t\, g_{K_t}$. Then
\[
\mathrm{II}_t(\nabla_t f_0,\nabla_t f_0)
\ge
\kappa_t \norm{\nabla_t f_0}^2 ,
\]
and therefore:
\[
\int_{K_t}\mathrm{II}_t(\nabla_t f_0,\nabla_t f_0)\,dA
\ge
\kappa_t\int_{K_t}\norm{\nabla_t f_0}^2\,dA
=
\kappa_t\frac{\lambda}{f^2}
\int_{K_t}f_0^2\,dA .
\]
\end{proof}



\noindent{\bf  Funding}
J. Mel\'endez is supported by the 2025 Special Program for Teaching and Research Project Funding at CBI-UAMI.


\nocite{*}
\bibliographystyle{acm}

\end{document}